\documentclass[preprint,12pt]{elsarticle}

\usepackage{amsmath,amsthm,amscd,amssymb}
\usepackage{enumerate}
\usepackage{xcolor}
\usepackage{algorithm}
\usepackage{algpseudocode}

\newtheorem{theorem}{Theorem}
\newtheorem{proposition}{Proposition}
\newtheorem{lemma}{Lemma}
\newtheorem{corollary}{Corollary}

\newcommand*{\gp}[1]{\langle\;#1\;\rangle}
\newcommand*{\order}[1]{\vert #1 \vert}
\newcommand*{\cd}{\circledast}

\DeclareMathOperator{\imf}{im}

\begin{document}
	
\begin{frontmatter}

	\title{The Order of the Unitary Subgroups of Group Algebras}

	\author{Zsolt Adam Balogh}
	\ead{baloghzsa@gmail.com}

	\address{Department of Mathematical Sciences\\
			 United Arab Emirates University, Al Ain, \\
			 United Arab Emirates, P.O.Box: 15551}
		
\begin{abstract}
Let $FG$ be the group algebra of a finite $p$-group $G$ over a finite field $F$ of positive characteristic $p$. Let $\cd$ be an involution of the algebra $FG$ which is a linear extension of an anti-automorphism of the group $G$ to $FG$.
If $p$ is an odd prime, then the order of the $\cd$-unitary subgroup of $FG$ is established.
For the case $p=2$ we generalize a result obtained for finite abelian $2$-groups. It is proved that the order of the $*$-unitary subgroup of $FG$ of a non-abelian $2$-group is always divisible by a number which depends only on the size of $F$, the order of $G$ and the number of elements of order two in $G$. 
Moreover, we show that  the order of the $*$-unitary subgroup of $FG$ determines the order of the finite
$p$-group $G$.
\end{abstract}
			
		
\end{frontmatter}

\section{Introduction and main results}

Let $FG$ be the group algebra of a finite $p$-group $G$ over a finite field $F$ of positive characteristic $p$. Let
\[
V(FG)=\Big\{\; x=\sum_{g\in G}\alpha_gg\in FG\; \mid\;  \chi(x)=\sum_{g\in G}\alpha_g=1\; \Big\}
\]
be the group of normalized units of $FG$, where  $\chi(x)$ is the augmentation map (see \cite [Chapters 2-3, p.\,194-196]{Bovdi_survey}). In this case, the order of the group  $V(FG)$ is equal to $\order{F}^{\order{G}-1}$, so  the order of  $V(FG)$ can be very large even for a small group $G$. Note that, studying the structure of the group $V(FG)$ is a rather difficult task (for more details see the survey \cite{Bovdi_survey}).

Let $\cd$ be an involution of the algebra $FG$. We say that the involution $\cd$ arises from the group $G$, if $\cd$ is a linear extension of an anti-automorphism of $G$ to $FG$. An example for such kind of involution  is the canonical involution that is the linear extension of the anti-automorphism of $G$ which sends each element of $G$ to its inverse. This involution is usually denoted by $*$.

An element $u \in V(FG)$ is called {\it $\cd$-unitary}, if $u^{\cd}=u^{-1}$ with respect to the involution $\cd$ of $FG$.
The set $V_{\cd}(FG)$ of all $\cd$-unitary units  forms a subgroup of $V(FG)$ which is called $\cd$-unitary subgroup.
Interest in the structure of unitary subgroups arose in algebraic topology and unitary $K$-theory (see Novikov's papers \cite{Novikov} and  Bovdi's paper \cite{Bovdi_Unitarity}). Let $L$ be a finite Galois extension of $F$ with Galois group $G$, where $F$ is a finite field of characteristic two.
Serre \cite{serre} identified an interesting relation between the self-dual normal basis of $L$ over $F$ and the $*$-unitary subgroup of $FG$. This relationship also makes the study of the unitary subgroups timely.

The unitary subgroups have been proven to be very useful subgroups in several studies (see \cite{Balogh_MD, Balogh_Creedon_Gildea, Balogh_Laver, Bovdi_Erdei_II, Bovdi_Erdei_I, Bovdi_Szakacs_II,  Bovdi_Kovacs_I, Bovdi_Salim, Bovdi_Grichkov, Creedon_Gildea_I} and \cite{Creedon_Gildea_II}). However, we know very little about their structure, as even finding their order is a challenging problem. The first results in this area were published in the $1980$'s. For finite abelian $p$-groups $G$ the order of $V_*(FG)$ was given in \cite{Bovdi_Szakacs_III}.
\begin{proposition}(\cite[Theorem 2]{Bovdi_Szakacs_III})\label{szakacs}
	Let $G$ be a finite abelian $2$-group. If ${F}$ is  a finite field of characteristic $2$, then the order  $\order{V_*({F}G)}$ is divisible by $\order{F}^{\frac{1}{2}(\order{G}+\order{G\{2\}})-1}$, that is, 	
    \[
       \order{V_*({F}G)}=\Theta \cdot \order{F}^{\frac{1}{2}(\order{G}+\order{G\{2\}})-1},
	\]
	where $\Theta=\order{G^2\{2\}}$ and $\order{S}$ denotes the size of a finite set $S$.
\end{proposition}

It follows that the number $\Theta$ does not depend on the size of the field $F$. The following breakthrough result was proved for certain non-abelian $2$-groups by Bovdi and Roza.
\begin{proposition}(\cite[Corollary 2]{Bovdi_Rosa_I})\label{roza}
	If $\order{F}=2^m\geq 2$, then:
	\[
	\order{V_*({F}G)}=\Theta \cdot \order{F}^{\frac{1}{2}(\order{G}+\order{G\{2\}})-1},
	\]
	where
	\begin{enumerate}
		\item[(i)] $\Theta=1$ if $G$ is a dihedral $2$-group;
		\item[(ii)] $\Theta=4$ if $G$ is a generalized quaternion $2$-group.
	\end{enumerate}
\end{proposition}

In \cite{Balogh_IEJA}, the value of the number  $\Theta$ was given for all non-abelian groups of order $2^4$. Although, for these groups, the number  $\Theta$ is not equal to $\order{G^2\{2\}}$ it does not depend on the field $F$. Wang and Liu \cite{Wang} evaluated   $\Theta$ in the case when $G$ is a non-abelian $2$-group given by a central extension of the form
\[
1 \longrightarrow C_{2^m} \longrightarrow G \longrightarrow C_2 \times \cdots \times C_2 \longrightarrow 1,
\]
in which $m\geq 1$ and $|G'|=2$. At present, the question of whether the quotient $\Theta$ depends on the field $F$ is still open.

Our main results are  the following.
\begin{theorem}\label{main_theorem_p}
	Let $G$ be a finite $p$-group, where $p$ is an odd prime and  let  $F$ be a finite field of characteristic $p$. If  $\cd$ is  an involution of $FG$ that arises from the group $G$, then
    \[
       \order{V_{\cd}(FG)}=\order{F}^{\frac{1}{2}(\order{G}-\order{G_{\cd}})},
    \]
    where $G_{\cd}=\{\,g\; \vert \; g=g^{\cd}\,\}$.
\end{theorem}

Let $\xi(G)$ denote the center of the group $G$ and $\xi(G)\{2\}$ denote the set of elements of order two in $\xi(G)$.
\begin{theorem}\label{main_theorem}
	Let $G$ be a finite $2$-group. If  $F$ is  a finite field of characteristic two, then
	\[
	\order{V_*({F}G)}=\Theta \cdot \order{F}^{\frac{1}{2}(\order{G}+\order{G\{2\}})-1}
	\]
	for some integer $\Theta$.
	Moreover, if the set $T_c=\{g\in G\;\vert\;g^2=c\}$ is  commutative for some $c\in \xi(G)\{2\}$, then $\Theta$ does not depend on the field $F$.
\end{theorem}

By combining Theorem \ref{main_theorem_p} and Theorem \ref{main_theorem}, we have  the following.

\begin{corollary}\label{C:1}
Let $G$ be a finite $p$-group. If  $F$ is  a finite field of characteristic $p$,  then the order of the $*$-unitary subgroup of $FG$ determines the order of $G$.
\end{corollary}

\section{Proofs}

Let $G$ be a finite $p$-group,  let $F$ be a finite field of $char(F)=p>2$ and  let $\cd$ be an involution of $FG$ which arises from $G$. An element $x\in FG$ is called skew-symmetric under the involution $\cd$ if $x^{\cd}=-x$. Let $FG^-_{\cd}$ denote the set of all skew-symmetric elements of $FG$.

\begin{proof}[Proof of Theorem \ref{main_theorem_p}]
Let  $z\in FG$ such that $1+z$ is invertible. Clearly, $1-z$ and $1+z$ commute, therefore $1-z$ and $(1+z)^{-1}$ also commute.

Let $Q=\{x\in FG\;|\; 1+x \;\text{is invertible in $FG$}\}$.
	Let us define the map $f:Q \rightarrow FG$ by
	\[
	f(x)=(1-x)(1+x)^{-1}.
	\]

If   $y\in FG^-_{\cd}$, then  $\chi(y)=0$, so 
\[
\chi(1+y)=\chi(1-y)=\chi(1+y^{\cd})=1
\]
 and $1+y, 1-y, 1+y^{\cd}$ are normalized units. Hence
\[
\begin{split}
f(y)f(y)^{\cd}=&(1-y)(1+y)^{-1}(1+y^{\cd})^{-1}(1-y^{\cd})\\
&=(1-y)(1+y)^{-1}(1-y)^{-1}(1+y)\\
&=(1+y)^{-1}(1-y)(1-y)^{-1}(1+y)\\
&=1.
\end{split}
\]
Consequently, $f(y) \in V_{\cd}(FG)$ and $f: FG^-_{\cd}\rightarrow V_{\cd}(FG)$ is a surjection.
	
Let  $x\in V_{\cd}(FG)$. Evidently, $1+x,1+x^{\cd}$ and $1+x^{-1}$ are invertible, because $\chi(1+x)=\chi(1+x^{\cd})=\chi(1+x^{-1})=2$ is invertible in $F$. Therefore $V_{\cd}(FG)$ is a subset of $Q$. Let $y$ denote the element $f(x)$. Then
	\[
	\begin{split} y^{\cd}&=f(x)^{\cd}=(1+x^{\cd})^{-1}(1-x^{\cd})=(1+x^{-1})^{-1}(1-x^{-1})\\
	&=\big(x^{-1}(x+1) \big)^{-1}x^{-1}(x-1)=-(1+x)^{-1}xx^{-1}(1-x)\\
&=-y.
	\end{split}
	\]
	Therefore $f:V_{\cd}(FG) \rightarrow FG^-_{\cd}$ is a surjection.

	Similar computation shows that
	\[
	\begin{split}
	 f(f(x))=&\big(1-(1-x)(1+x)^{-1}\big)\big(1+(1-x)(1+x)^{-1}\big)^{-1}\\ &=\big((1+x)-(1-x)\big)(1+x)^{-1}\Big(\big((1+x)+(1-x)\big)(1+x)^{-1}\Big)^{-1}\\
	&=\big((1+x)-(1-x)\big)(1+x)^{-1}  (1+x) \big((1+x)+(1-x)\big)^{-1}\\
	&=\big((1+x)-(1-x)\big)\big((1+x)+(1-x)\big)^{-1}\\
&=x\\
	\end{split}
	\]
	for every $x\in V_{\cd}(FG)$, so $f$ is a bijection between
	$FG^-_{\cd}$ and $V_{\cd}(FG)$.
	
Since $FG^-_{\cd}$ is a linear space over $F$ with basis
	$\{\,g-g^{\cd}\;\vert \; g\in G\setminus G_{\cd}\,\}$, \[
\order{V_{\cd}(FG)}=\order{FG^-_{\cd}}=\order{F}^{\frac{1}{2}(\order{G}-\order{G_{\cd}})}.
\]
\end{proof}

Let $H$ be a normal subgroup of $G$ and let $I(H):=\gp{1+h \;\mid \; h\in H}_{FG}$ be an ideal of $FG$ generated by the set $\{ 1+h \mid  h\in H\}$.
Clearly,
\[
FG/I(H)\cong FG/\ker(\Psi) \cong F[G/H],
\]
where $\Psi: FG \to  FG/I(H)$ is the natural homomorphism.

Let us denote by $V_*(F\overline{G})$ the $*$-unitary subgroup of the factor algebra $FG/I(H)$, where $\overline{G}=G/H$.
It is easy to check that the set
\[
N^*_{\Psi}=\{x\in V(FG) \mid  \Psi(x) \in V_*(F\overline{G}) \}
\]
forms a subgroup in $V(FG)$.
Let $I(H)^+=\{1+x \, \vert\, x\in I(H) \}$. The subgroup  $I(H)^+$ is  normal in $V(FG)$ and
$S_H=\{ xx^* \mid  x\in N^*_{\Psi} \}$ is a subset of $I(H)^+$, because  $xx^* \in 1+\ker(\Psi)=I(H)^+$ for all $x\in N^*_{\Psi}$.

First, we need the following.

\begin{lemma}\label{lemma_main}
	Let $H$ be a normal subgroup of a finite $2$-group $G$. Set $\overline{G}=G/H$. If  $\order{F}=2^m\geq 2$, then
	\begin{equation}
	\textstyle
	\order{V_*(FG)}=\order{F}^{\order{\overline{G}}}\cdot \frac{\order{V_*(F\overline{G})}}{\order{S_H}}.
	\end{equation}
\end{lemma}

\begin{proof}
	Let $\Phi: V(FG) \to V(FG)$ be a map such that  $\Phi(x)=xx^*$ for every $x\in V(FG)$.
	The sets $\Phi(x)$ and $\Phi(y)$ coincide if and only if $y \in x\cdot V_*(FG)$. Indeed, if $y \in x \cdot V_*(FG)$, then $y=xv$ for some $v \in V_*(FG)$. Therefore
	\[
	\Phi(y)=yy^*=xv(xv)^*=xvv^*x^*=xx^*=\Phi(x).
	\]
	Assume that $\Phi(x)=\Phi(y)$ for some $x,y\in V(FG)$. Then $xx^*=yy^*$, or equivalently, $y^{-1}x=y^*(x^*)^{-1}$. Therefore
	\[
	(x^{-1}y)^{-1}=y^{-1}x=y^*(x^*)^{-1}=(x^{-1}y)^*
	\]
	which confirms that $x^{-1}y\in V_*(FG)$.
	
	Since  $[N^*_{\Psi}:V_*(FG)] = \order{S_H}$,
	\[
	\textstyle
	 \order{V_*(FG)}=\frac{\order{N^*_{\Psi}}}{\order{S_H}}=\order{I(H)^+}\cdot \frac{\order{V_*(F\overline{G})} }{\order{S_H}}.
	\]
	We should note that $V_*(FG)$ is usually not a normal subgroup of $N^*_{\Psi}$.
	
	The ideal $I(H)$ can be considered as a vector space over $F$ with the following basis $\{\, u(1+h) \mid  u \in T(G/H),\; h \in H \,\}$, where $T(G/H)$ is a complete set of left coset representatives of $H$ in $G$. Consequently,
\[
\order{I(H)^+}=\order{I(H)}=\order{F}^{\frac{\order{G}}{\order{H}}}.
\]
\end{proof}

\begin{proof}[Proof of Theorem \ref{main_theorem}]
	Let $G$ be a $2$-group of order $2^n$ and let $H$ be a subgroup of $G$ generated by a central element $c$ of order two.
	Evidently, the set $S_H=\{\, xx^* \mid  x\in N^*_{\Psi}\}$ is a subset of
	$I(H)^+ \cap V_*(FG)$
	and every $y\in S_H$ is a $*$-symmetric element. Moreover, the support of $y$ does not contain elements of order two by \cite[Lemma 2.5]{Balogh_IEJA}.
	Thus,  every $y\in S_H$ can be written as
	\[
	y=1+\sum_{g\in G\setminus (G\{2\} \cup T_c)} \alpha_g (g+g^{-1})\widehat{H}+ \sum_{g\in T_c} \beta_g g\widehat{H},\qquad (\alpha_g, \beta_g \in F)
	\]
	where $T_c=\{\,g\in G\;\vert \;g^2=c\}$ and  $\widehat{H}=1+c$. This yields that
	\begin{equation}\label{ineqush1}
	\order{S_H} \leq \order{F}^{\frac{1}{4}(\order{G}-\order{G\{2\}}-\order{T_c})+\frac{1}{2}\order{T_c}}=\order{F}^{\frac{1}{4}(\order{G}-\order{G\{2\}}+\order{T_c})}.
	\end{equation}
Let us  prove that if $T_c$ is a commutative set, then
	\begin{equation}\label{ineqush2}
	\order{F}^{\frac{1}{4}(\order{G}-\order{G\{2\}}+\order{T_c})}\cdot 2^{-\frac{1}{2}\order{T_c}} \leq \order{S_H}.
	\end{equation}
Let $N_1$ be a group generated by the elements $1+\alpha_g (g+g^{-1})\widehat{H}$,  in which  $g^2\not\in H$ and $\alpha_g \in F$.
Evidently, $N_1$ is an elementary abelian subgroup of $I(H)^+$.
Since $g^2\not\in H$  (equivalently $g\in G\setminus (G\{2\} \cup T_c)$), 
\[
1+\alpha_g (g+g^{-1})\widehat{H}=1+\alpha_g g^{-1}(1+g^2)\widehat{H}\not=1
\]
and
	\[
	1+\alpha_g (g+g^{-1})\widehat{H}=(1+\alpha_g g\widehat{H})(1+\alpha_g g\widehat{H})^*\in S_H.
	\]
If $z\in N_1$, then
	\[
	z=\prod_{g^2\not\in H} (1+\alpha_g (g+g^{-1})\widehat{H})=\prod_{g^2\not\in H}(1+\alpha_g g\widehat{H})(1+\alpha_g g\widehat{H})^*,
	\]
so  $(1+\alpha_g g\widehat{H})\in I(H)^+$ and
	\[
	z=\Big(\prod_{g^2\not\in H}(1+\alpha_g g\widehat{H})\Big) \cdot \Big(\prod_{g^2\not\in H}(1+\alpha_g g\widehat{H})\Big)^*\in S_H.
	\]
	Thus $N_1$ is a subgroup in the set $S_H$ and $\order{N_1}=\order{F}^{\frac{1}{4}(\order{G}-\order{G\{2\}}-\order{T_c})}$. 	 
	
	The map $\tau: F\to F$ defined by $\tau(\alpha)=\alpha+\alpha^2$ $(\alpha \in F)$ is a homomorphism on the additive group of the field $F$ with kernel $\ker(\tau) =\{0,1\}$ (see \cite[Lemma 10]{Balogh_IEJA}). Therefore $\order{\imf(\tau)}=\frac{\order{F}}{2}$.
	
Suppose that $T_c$ is a commutative set.
Let $N_2$ be a group generated by the elements	 $1+\alpha_g g\widehat{H}$, in which  $g\in T_c$ and $\alpha_g \in F$.
	Since
	\[
	(1+\omega g+\omega g^2)(1+\omega g+\omega g^2)^*=1+(\omega+\omega^2) g \widehat{H}
	\]
	we have $1+\alpha_g g\widehat{H}\in S_H$ for every $\alpha_g \in \imf(\tau)$.
	The group $N_2$, being $T_c$ commutative, is a subgroup in $S_H$ and \[
	 \order{N_2}=\order{\imf(\tau)}^{\frac{1}{2}\order{T_c}}=\order{F}^{\frac{1}{2}\order{T_c}}\cdot 2^{-\frac{1}{2}\order{T_c}}.
	\]
	Let $x\in N_1$ and $y\in N_2$. There exist $x_1\in I(H)^+$ and $y_1\in N^*_{\Psi}$ such that $x_1x_1^*=x$ and $y_1y_1^*=y$. Since $I(H)^+$ is an elementary $2$-group, the element  $y$ commutes with $x_1$ and
	\[
	yx=yx_1x_1^*=x_1yx_1^*=x_1y_1y_1^*x_1^*=(x_1y_1)(x_1y_1)^*\in S_H.
	\]
	Therefore $N_1\times N_2$ is a subgroup in $S_H$ and
	\[
	\order{N_1\times N_2}=\order{F}^{\frac{1}{4}(\order{G}-\order{G\{2\}}-\order{T_c})+\frac{1}{2}\order{T_c}}\cdot 2^{-\frac{1}{2}\order{T_c}}.
	\]
	Consequently,
	\[
	\order{F}^{\frac{1}{4}(\order{G}-\order{G\{2\}}+\order{T_c})}\cdot 2^{-\frac{1}{2}\order{T_c}} \leq \order{S_H}.
	\]
	
	Now, we are ready to prove the theorem. If $n=3$, then the theorem is true by Propositions \ref{szakacs} and \ref{roza}.
	Suppose that $n>3$. In the factor group $\overline{G}=G/H$ the element $\overline{g}$ has order two
	if and only if  either $g\in G\{2\}$ or $g \in T_c$. Therefore, $\order{\overline{G}\{ 2 \}} = \frac{\order{G\{2\}}+\order{T_c}}{2}$.
	According to the inductive hypothesis
	\[\order{V_*(F\overline{G})}=2^s\cdot \order{F}^{\frac{1}{2}(\order{\overline{G}}+\order{\overline{G}\{2\}})-1}=2^s\cdot \order{F}^{\frac{1}{4}(\order{{G}}+\order{{G}\{2\}}+\order{T_c})-1}
	\]
	for some $s \geq 0$. Using  Lemma \ref{lemma_main} and equation (\ref{ineqush1}) we obtain that
	\[
	\begin{split}
	\order{V_*(FG)}&=\textstyle\order{F}^{\frac{\order{G}}{2}}\cdot \frac{\order{V_*(F\overline{G})}}{\order{S_H}}\\
&\textstyle \geq \order{F}^{\frac{\order{G}}{2}}\cdot  \frac{2^s\cdot \order{F}^{\frac{1}{4}(\order{G}+\order{G\{2\}}+\order{T_c})-1}}	 {\order{F}^{\frac{1}{4}(\order{G}-\order{G\{2\}}+\order{T_c})}}\\
	&=2^s\cdot \order{F}^{\frac{1}{2}(\order{G}+\order{G\{2\}})-1}.
	\end{split}
	\]
	Therefore
	\begin{equation}\label{inequV}
	2^s\cdot \order{F}^{\frac{1}{2}(\order{G}+\order{G\{2\}})-1} \leq \order{V_*(FG)}
	\end{equation}
	and $\order{V_*(FG)}$ is divisible by $\order{F}^{\frac{1}{2}(\order{G}+\order{G\{2\}})-1}$.

	Similarly, using  Lemma \ref{lemma_main} and  (\ref{ineqush2}), we obtain that
	\[
	\begin{split}
	\order{V_*(FG)}&=\textstyle\order{F}^{\frac{\order{G}}{2}}\cdot \frac{\order{V_*(F\overline{G})}}{\order{S_H}}\\
 &\textstyle \leq \order{F}^{\frac{\order{G}}{2}}\cdot  \frac{2^s\cdot \order{F}^{\frac{1}{4}(\order{G}+\order{G\{2\}}+\order{T_c})-1}}	 {\order{F}^{\frac{1}{4}(\order{G}-\order{G\{2\}}+\order{T_c})}\cdot 2^{-\frac{1}{2}\order{T_c}}}\\
	&=2^{s+\frac{1}{2}\order{T_c}}\cdot \order{F}^{\frac{1}{2}(\order{G}+\order{G\{2\}})-1}.
	\end{split}
	\]
	
	The size of the set $T_c$ does not depend on the field $F$. Since $2^s$ does not depend on the field $F$ by the inductive hypothesis,  the proof is complete.
\end{proof}

We should remark that $2^s$ in the inequality ($\ref{inequV}$) is usually not inherited via the factorization. For example, for the semidihedral group $D_{16}^-$ of order $16$ we have that
\[
\order{V_*(FD_{16}^-)}=2\cdot \order{F}^{\frac{1}{2}(\order{D_{16}^-}+\order{D_{16}^-\{2\}})-1}
\]
 by \cite[Lemma $3.4$]{Balogh_IEJA}. However, $D_{8}\cong D_{16}^-/H$, where $H$ is the center of $D_{16}^-$ and $D_{8}$ is the dihedral group of order $8$ and
$\order{V_*(FD_{8})}=\order{F}^{\frac{1}{2}(\order{D_{8}}+\order{D_{8}\{2\}})-1}$ by Proposition \ref{roza}.

\begin{proof}[Proof of Corollary \ref{C:1}]
If $p=2$, then  Theorem \ref{main_theorem} implies that
	\[
	\order{F}^{\frac{\order{G}}{2}-1} \leq \order{F}^{\frac{1}{2}(\order{G}+\order{G\{2\}})-1} \leq \order{V_*(FG)} \leq \order{F}^{\order{G}-1}.
	\]	
	If  $p$ is an odd prime, then $\order{V_*(FG)}=\order{F}^{\frac{1}{2}(\order{G}-1)}$ by  Theorem \ref{main_theorem_p}. Hence, $\order{F}^{\frac{\order{G}}{2}-1} \leq \order{V_*(FG)} \leq \order{F}^{\order{G}-1}$, which  confirms that the order of $V_*(FG)$ determines the order of $G$ for every finite $p$-groups.
\end{proof}

\bibliographystyle{abbrv}

\end{document}